\documentclass[a4paper,10pt]{amsart}
 
\usepackage[english]{babel}
\usepackage{dsfont} 
\usepackage{graphicx}
\usepackage{color}
\usepackage{enumerate}
\usepackage{array,multirow}

\allowdisplaybreaks

\usepackage{amssymb} 
\usepackage{amsthm} 
\usepackage{dsfont}

\newtheorem{thm}{Theorem}
\newtheorem{dfn}[thm]{Definition}
\newtheorem{lem}[thm]{Lemma}

\newtheorem{exm}[thm]{Example}

\newtheorem{prop}[thm]{Proposition}
\newtheorem{rem}[thm]{Remark}
\newtheorem{cor}[thm]{Corollary}

\newcommand{\dK}{\mathds{K}}
\newcommand{\dN}{\mathds{N}}

\newcommand{\dR}{\mathds{R}}
\newcommand{\dC}{\mathds{C}}
\newcommand{\ii}{\mathrm{i}}

\newcommand{\cL}{\mathcal{L}\,}
\newcommand{\cD}{\mathcal{D}\,}
\newcommand{\cH}{\mathcal{H}\,}
\newcommand{\Lin}{\mathrm{Lin}\,}

\newcommand{\her}{\mathrm{her}\,}

\newcommand{\cA}{\mathcal{A}}
\newcommand{\cB}{\mathcal{B}}

\newcommand{\cN}{\mathcal{N}}

\newcommand{\cS}{\mathcal{S}}

\author{Konrad Schm\"udgen}
\address{University of Leipzig, Mathematical Institute, Augustusplatz 10/11, D-04109 Leipzig, Germany}
\email{\tt schmuedgen@math.uni-leipzig.de}

\begin{title}{$*$-Bimodules}
\end{title}

\begin{document}

\begin{abstract}
A  $*$-bimodule for a unital $*$-algebra $A$ is an $A$-bimodule $X$ which is  a vector space with involution $x\mapsto x^+$ satisfying $(a\cdot x\cdot b)^+=b^+\cdot x^+\cdot b^+$ for $x\in X$ and $a,b\in A$. An algebraic model for $*$-bimodules is given. Hilbert space representations of $*$-bimodules are defined and studied. A  GNS-like representation theorem is obtained.
\end{abstract}

\maketitle

\textbf{AMS  Subject  Classification (2000)}.
 46L50, 47L60.\\

\textbf{Key  words:} bimodule, algebra with involution, unbounded representation

\section{Introduction}

This paper deals with an algebraic structure that seems to be not yet studied in the literature (at least as far as the author is aware): {\it $*$-bimodules} of general $*$-algebras.

Let $A$ be a real or complex unital $*$-algebra with involution denoted by $a\mapsto a^+$. A {\it $*$-bimodule} for $A$ is an $A$-bimodule $X$ which is also areal or complex vector space with involution $x\mapsto x^+$ satisfying the compatibility condition
\begin{align}
(a\cdot x\cdot b)^+=b^+\cdot x^+\cdot b^+, \quad {\rm for}~~~x\in X,~ a,b\in A,
\end{align}
where the dot $\cdot$ stands for the left and right actions of $A$ on $X$.

For instance, if $A$ is the polynomial $*$-algebra $\dC[x_1,\dots,x_d]$ and $f_1,\dots,f_k$ are real-valued functions on $\dR^d$, then $X:=f_1\cdot A+\cdots+f_k\cdot A$ is obviously a $*$-bimodule for $A$, with multiplication of functions as left and right actions and complex conjugation as involution.

The main aim of this paper is to define and discuss notions of Hilbert space representations of $*$-bimodules and to develop a GNS-like construction. 

This paper is organized as follows. Section \ref{defex} contains general definitions and a large number of examples of $*$-bimodules.
In Section \ref{algreal}, we develop an algebraic model for $*$-bimodules. We show that each left action of $A$ on a vector space $D$ yields a $*$-bimodule structure on the vector space $B(D,D)$ is sesquilinear forms on $D$. 
Section \ref{Hilbrep} is the heart of this paper. We give three possible definitions of  $*$-representations of  $*$-bimodules on  inner product spaces. 
The main result of Section \ref{gnslike}  can be interpreted as an extension of the GNS construction to $*$-bimodules. 
In the final Section \ref{anexampleweyl} we elaborate some examples of the GNS-like construction.

A large part of this  paper deals with unbounded Hilbert space representations of $*$-algebras and  general $*$-algebraic constructions.
For $*$-representations of  $*$-algebras we refer to the author's books \cite{schm1990}, \cite{schm2020}, for $*$-representations of quasi $*$-algebras to \cite{tfr} and for $*$-representations of partial $*$-algebras to the monograph \cite{ait}.

Let us recall the notion of a $*$-representation of a $*$-algebra $A$. Let $(\cD,\langle\cdot,\cdot\rangle)$ be a complex inner product  space. A {\it $*$-representation} of  $A$ on $\cD$ is an algebra homomorphism $\rho$ of $A$ into the algebra of linear operators on $\cD$ such that
\begin{align*}
\langle \rho(a)\varphi,\psi\rangle=\langle\varphi,\rho(a^+)\psi\rangle, \quad a\in A,~ \varphi,\psi\in \cD.
\end{align*}
If $A$ is unital with unit $1$, then $\rho$  is called {\it nondegenerate} if $\rho (1)\varphi=\varphi$ for all $\varphi\in \cD$.
\section{Definition and Examples of $*$-Bimodules}\label{defex}

We begin with   a number of standard definitions. All vector spaces and algebras are either over the real field or over the complex field. Let $\dK$ denote $\dC$ or $\dR$.

A  {\it left $A$-module} is a vector space $X$ over   $\dK$ equipped with a bilinear mapping $(a,x)\mapsto a\cdot x$ of $A\times X$ in $X$, called then a left action of $A$ on $X$, such that
 \begin{align*}
ab\cdot x=a\cdot (b\cdot x), \quad 1\cdot x=x \quad {\rm for}~~~ a,b\in A, x\in X.
\end{align*}
Similarly, a vector space $X$ is   {\it right $A$-module} if $X$ is equipped with a mapping $(x,a)\mapsto x\cdot $ of $X\times A$ in $X$, called a right action of $A$ on $X$, such that
 \begin{align*}
x\cdot (ab)=(x\cdot a) \cdot b, \quad x\cdot 1=x \quad {\rm for}~~~ a,b\in A, x\in X.
\end{align*}
If  no ambiguity can arise, the symbol $\cdot$ is used to denote left or right actions. If  we want to distinguish between different actions, we use a symbol such as $\rho$ to write $\rho(a)x=a \cdot x$ or $x\, \rho(a)=x \cdot a$, respectively.

An {\it $A$-bimodule} is a vector space  $X$ which is a left $A$-module and a right $A$-module such that the corresponding left and right actions commute, that is, we have 
\begin{align}\label{combimod}
(a\cdot x )\cdot b=a\cdot (x \cdot b) \quad {\rm for}~~~ a,b\in A, x\in X.
\end{align}
For an $A$-bimodule $X$, we can omit the brackets and write simply $a\cdot x \cdot b$ for the expression $a\cdot (x \cdot b) =(a\cdot x )\cdot b$.

\begin{rem} In the standard algebra literature (see, e.g., \cite{lam1999}), bimodules of rings are defined as abelian groups with  left and right module structures satisfying (\ref{combimod}). $*$-Bimodules for $*$-rings can be also defined and treated in this setting. However, since our main motivation lies in Hilbert space representations of $*$-algebras, we prefer to work with vector spaces as bimodules and $*$-bimodules. 
\end{rem}
A {\it $*$-vector space} is a vector space $X$ together with an {\it involution} $x \mapsto x^+$ of $X$ into itself, that is,
\begin{align}
(\alpha  x +\beta y)^+=\overline{\alpha}x^+ + \overline{\beta} y^+~~~ {\rm and}~~~ (x^+)^+=x ~~~ {\rm for}~~~ \alpha,\beta\in\dK,~~   x, y\in X.
\end{align}
For a   $*$-vector space $X$ its {\it hermitian part} $X_{\her}:=\{ x\in X: x=x^+\}$ is a real vector space. In the case $\dK=\dC$ each element $x\in X$ can uniquely written as $x=x_1+\ii x_2$ with $x_1,x_2\in X_{\her}$.

A {\it $*$-algebra} $A$ is an algebra over $\dK$, equipped with an involution $a\mapsto a^+$ such that $A$ is a $*$-vector space and
\begin{align}
(ab)^+=b^+a^+ \quad {\rm for}~~~a,b,\in A.
\end{align}
If $A$ is   unital, the unit element is always denoted by $1$ and it satisfies $1^+=1$.

The involution of a $*$-vector space can be used to define  a left action from a right action and vice versa, as the following simple lemma shows.
\begin{lem}
 Suppose $X$ is a $*$-vector space which is a right $A$-module. Then $X$ becomes a left $A$-bimodule with left action defined by
\begin{align}\label{leright}
a\cdot x:=(x^+\cdot a^+)^+ ,\quad a\in A, x\in X.
\end{align}
\end{lem}
\begin{proof}
We have $1\cdot x=(x^+ \cdot 1^+)^+= (x^+)^+=x$ and
\begin{align*}
\hspace{3.1cm}a\cdot[b\cdot x]&=a \cdot [(x^+\cdot b^+)^+]= (((x^+\cdot b^+)^+)^+ \cdot a^+)^+\\&=((x^+\cdot b^+) \cdot a^+)^+=(x^+\cdot b^+ a^+)^+
\\ &=(x^+\cdot (ab)^+)^+ =(ab)\cdot x. \qedhere 
\end{align*}
\end{proof}

\smallskip

The following definition introduce the main concept of which this paper is about.
\begin{dfn} A \emph{$*$-bimodule} for $A$ is an $A$-bimodule $X$ which is a $*$-vector space with involution $x\to x^+$ satisfying the compatibility condition
\begin{align}\label{starcon}
(a\cdot x\cdot b)^+ =b^+\cdot x^+\cdot a^+ \quad{\rm for}~~~x\in X, a,b\in A.
\end{align}
\end{dfn}
Note that for a $*$-bimodule $X$ the left action of $A$ is obtained from the right action of $A$ and the involutions of $A$ and $X$ by formula (\ref{leright}). This simple observation is very useful for many considerations.

The simplest $*$-bimodule $X$  is, of course, the $*$-algebra $A$ itself, with left and right multiplications as left and right actions, respectively. 
Another special case is  described in the following definition.
\begin{dfn}\label{quasidef}
A \emph{quasi-$*$-algebra}\footnote{Quasi $*$-algebras were introduced by G. Lassner \cite{las84} with a different definition. He defined {\it topological quasi-$*$-algebras}. This means that $X$ is equipped with a locally convex topology such that the algebraic operations are continuous and $A$ is dense in $X$, see \cite[Definition 3.4.7]{schm1990} or \cite[Definition 5.1.11]{tfr} for a precise formulation. The density condition is crucial for many considerations. It should be noted the the bimodule axiom (\ref{combimod}) is missing in \cite{las84}; it was added later in \cite{schm1990}. Quasi $*$-algebras as in Definition \ref{quasidef} are defined in \cite[Subsection 1.1]{tfr}.} is a pair $(A,X)$ of a unital $*$-algebra $A$ and a $*$-bimodule $X$ such that  $A\subseteq X$ and the actions  and the involution of elements of $A\subseteq X$ on $A$ are given by the multiplications  and the involution of $A$, respectively.
\end{dfn}

Thus, for quasi $*$-algebras we have the very special situation that the $*$-algebra $A$ is contained in the $*$-bimodule $X$. Of course,  this is not true in general.  Quasi $*$-algebras are much closer to $*$-algebras and most of the methods used in their study 
 do not apply  to $*$-bimodules. Nevertheless quasi $*$-algebras might be the first case to look at when we consider a problem for $*$-bimodules. An extensive treatment of quasi $*$-algebras and their representation theory is provided in the  recent monograph \cite{tfr} by C. Trapani and M. Fragoulopoulou.

Before we continue the general considerations let us give some examples.

\begin{exm}
Suppose $A$ is a commutative  unital $*$-algebra and  $c=c^+$ is an element of $A$ which is not a zero divisor. Then 
$X_1 :=\{ \frac{a}{c}:a\in A\}$ and $X_2:=A+X_1$ are $*$-bimodules for $A$ with multiplication of $A$ as left and right actions. Note that $X_1$ does not contain $A$ if $c$ is not invertible in $A$.
\end{exm}
\begin{exm}\label{exminusx2}
Suppose $A$ is the $*$-algebra
$\dC[x_1,\dots,x_d]$  with involution determined by $(x_j)^+=x_j, j=1,\dots,d$. Let $\alpha>0$. Clearly, $X=A\cdot e^{-\alpha (x_1^2+\cdots+x_d^2)}$ is a $*$-bimodule for $A$, again with multiplication as actions.
\end{exm}

A large variety of examples are obtained from the following general setup.
\begin{exm}\label{generalex}
Suppose $A$ is a $*$-subalgebra of a larger $*$-algebra $B$ and $Y$ is a subset of the hermitian part of $B$. Then the linear span $X$ of products $ayb$ in $B$, where $a,b\in A$ and $y\in Y$,  is a $*$-bimodule for $A$, with left and right actions of $A$ given by the left and right multiplications of $B$. The $A$-bimodule axioms and the compatibility condition (\ref{starcon}) for the involution follow at once from the $*$-algebra axiom for the larger $*$-algebra $B$.
\end{exm}

We illustrate  Example \ref{generalex} with two applications.
\begin{exm}\label{pnweyl}
Let $B$ be the one-dimensional \emph{Weyl algebra}, that is, $B$ is the unital complex $*$-algebra with two hermitian generators $p,q$ and defining relation
$$
pq-qp=-\ii \cdot 1.
$$ 
Let $A$ be the $*$-subalgebra $\dC[q]$ of $B$ and $n\in \dN_0$. Then, by Example \ref{generalex}, the vector space $X$ of finite sums $\sum_j a_j(q)p^nb_j(q)$, where $a_j,b_j\in \dC[q]$, is a $*$-bimodule for $A$.
\end{exm}

\begin{exm}\label{bofh}
Suppose $B$ is the $*$-algebra of bounded operators on a Hilbert space $\cH$ and  $A$ is a $C^*$-algebra on $\cH$. Let $Y$ be  a set of bounded self-adjoint operators on $\cH$. Let $X$ be the vector space of finite sums $\sum_j a_jy_jb_j$, where $a_j,b_j\in A$ and $y_j\in T$. By  Example \ref{generalex}, $X$ is a $*$-bimodule of $X$.
\end{exm}

We continue with a few general definitions. The first one is more or less obvious.
\begin{dfn}
We say that a subset $Y$ of a $*$-bimodule $X$ for $A$ \emph{generates} $X$ if $X$ is the smallest $*$-bimodule for $A$ which contains $Y$.
\end{dfn}
This means that $X$ is the linear span of elements $a\cdot y\cdot b$ and $a\cdot y^+\cdot$, where $a,b\in A$ and $y\in Y$.
The element $p^n$ generates the $*$-bimodule $X$ in Example \ref{pnweyl} and the set $Y$ in Example \ref{bofh} generates the $*$-bimodule $X$ defined therein.

The well known notion of a quadratic module for $*$-algebras  extend in a natural manner to $*$-bimodules.
\begin{dfn}
A \emph{quadratic module} of a $*$-bimodule $X$ is a subset $Q$ of the hermitian part $X_\her$ such that $x+y\in Q$ and $\lambda x\in Q$ for $x,y\in Q$, $\lambda\geq 0,$ and
\begin{align}
a^+\cdot x\cdot a\in Q \quad {\rm for}~~~x\in Q,~a\in A.
\end{align}
\end{dfn}

Note that for  $a\in A$ and $x\in X_\her$, the element $a^+\cdot x\cdot a$ always belongs to $X_\her$.

The possibility of defining quadratic modules  in $*$-bimodules leads to several natural questions: Can one develop  ideas,  notions and results of real algebraic geometry in the setup of $*$-bimodules?  In particular, can one prove  Positivstellens\"atze for appropriate quadratic modules  in $*$-bimodules, at least in the case of a commutative real algebra $A$?

However, the attempt to define  preorderings for $*$-bimodules leads to difficulties. This is similar to the case of noncommutative real algebraic geometry where the difficulties arise from the fact that products of noncommuting hermitian elements are not hermitian (see, e.g., \cite[Section 4.2]{schm2009} for a discussion of this matter). For $*$-bimodules it is even worse: In general it is impossible to form products of elements of $X$.
\begin{exm}
Let $X$ be a $*$-bimodule for $A$ and let $Y$ be a subset of $X_\her$. Then
\begin{align*}
Q=\Big\{ \sum_{j=1}^k\, (a_j)^+\cdot y_j\cdot a_j:~ a_j\in A, y_j\in Y, k\in \dN\Big\}
\end{align*}
is a quadratic module of $X$,  the smallest quadratic module containing the set $Y.$
\end{exm} 
\begin{exm}
Suppose  $A=\dR[x_1,\dots,x_d]$. Let $p_1,\dots,f_k$ be nonzero polynomials of $A$. Then $X:=p_1^{-1}A+\dots+p_k^{-1}A$ is a $*$-bimodule for $A$. 

Let $P$ be a quadratic module of $A$. Then, obviously, $Q= p_1^{-1}P+\dots+p_k^{-1}P$ is a quadratic module of the $*$-bimodule $X$. For instance, if $P=\sum A^2$, then $Q$ is a set of sums of elements $p_j^{-1}p^2$ with $p\in A$ and $j=1,\dots,k$.
\end{exm}
The next definition gives     another important notion for $*$-bimodules.
\begin{dfn}\label{homomprphism}
Let $X_1$ and $X_2$ be $*$-bimodules for unital $*$-algebras $A_1$ and $A_2$, respectively. A\,  \emph{homomorphism}\, of $X_1$ in $X_2$ is a pair $(\rho,\vartheta)$ of a unital $*$-homomorphism $\rho:A_1\mapsto A_2$ and a linear mapping $\vartheta:X_1\mapsto X_2$  such that
\begin{align}
\theta(x^+)=\theta(x)^+ ~~~ {\rm and}~~~ \theta(a\cdot x\cdot b)=\rho(a)\cdot \theta(x)\cdot \rho(b) \quad {\rm for}~~~ a,b\in A_1,\, x\in X_1.
\end{align}
\end{dfn}
In particular, this definition applies to $*$-bimodules $X_1$ and $X_2$ for the same $*$-algebra $A_1=A_2$.
\section{Algebraic representations of $*$-bimodules}\label{algreal}

In this section we consider $*$-bimodules in a  purely algebraic setting.

Let $D$ be a  vector space over $\dK$.
We denote by $B(D,D)$ the $\dK$-vector space of all sesquilinear forms $x(\cdot, \cdot)$ on $D \times D$, that is $x$ is a mapping of $D \times D$ into $K$ which is linear in the first and conjugate linear in the second variables. In the real case $B(D,D)$ is just the real vector space of bilinear forms over $D$.

Clearly, $B(D,D)$ is a $*$-vector space with the involution defined by \begin{align}\label{defiaction}
x^+(\varphi,\psi):=\overline{x(\psi,\varphi)}, \quad\varphi,\psi\in D.
\end{align}

Suppose that $D$  is a {\bf left $A$-module}. We write simply $a\varphi$ for the left action of $a\in A$ on $\varphi\in D$.

For $a,b\in A$ and $x\in B(D,D)$ we define two sesquilinear forms $a\cdot x$ and $x\cdot b$ from $B(D,D)$ by
\begin{align}\label{defin}
(a\cdot x)(\varphi,\psi)=x(\varphi,a^+ \psi), \quad (x\cdot b)(\varphi,\psi)=x(b\varphi,\psi),\quad \varphi,\psi\in \cD.
\end{align}
\begin{prop}\label{algexa}
With these definitions (\ref{defiaction}) and (\ref{defin}), the $*$-vector space $B(D,D)$ is a $*$-bimodule for $A$.
\end{prop}
\begin{proof}
From (\ref{defin}) we derive 
$$
(b \cdot (a \cdot x))(\varphi,\psi)=(a\cdot x)(\varphi, b^+\psi)=x(\varphi, a^+b^+\psi)=x(\varphi,(ba)^+\psi)= ((ba)\cdot x)(\varphi,\psi),
$$
that is, $b \cdot (a \cdot x)=(ba)\cdot x$. Moreover, we get $(1\cdot x)(\varphi, \psi)=x(\varphi,1\cdot \psi)=x(\varphi,\psi)$, so that $1\cdot x=x$. Thus, $B(D,D)$ is a left $A$-nodule with action $(a,x)\mapsto a\cdot x$ defined by (\ref{defin}). Similarly, $B(D,D)$ is a right $A$-module with right action $(x,a)\mapsto x\cdot a$.

By
(\ref{defin}) we obtain 
$$
((a\cdot x) \cdot b)(\varphi,\psi)=x(b\varphi, a^+\psi) \quad {\rm and} \quad (a\cdot (x \cdot b))(\varphi,\psi)=x(b\varphi, a^+\psi),
$$
so that $(a\cdot x) \cdot b=a\cdot(x\cdot b)$. Hence $B(D,D)$ is an $A$-bimodule. 

Next we verify the $*$-condition (\ref{starcon}). Since $(a\cdot x \cdot b)(\psi,\varphi)=x(b\psi, a^+\varphi)$, we have
\begin{align}\label{starcon}
(a\cdot x \cdot b)^+(\varphi,\psi)=\overline{(a\cdot x \cdot b)(\psi,\varphi)}=\overline{ x(b\psi, a^+\varphi)}.
\end{align}
On the other hand, from $x^+(\xi,\eta)=\overline{x(\eta,\xi)}$ we get 
\begin{align}\label{starcon1}
(b^+\cdot x^+ \cdot a^+)(\varphi,\psi)=x^+(a^+\varphi,b\psi)=\overline{x(b\psi, a^+\varphi)}.
\end{align}
Since the expressions in (\ref{starcon}) and (\ref{starcon1}) coincide,  we have $(a\cdot x \cdot b)^+=b^+\cdot x^+ \cdot a^+$. Thus we have shown that $B(D,D)$ is a $*$-bimodule. \end{proof}
An immediate consequence of Proposition \ref{algexa} is the following.

\begin{cor}\label{substar}
Suppose $X$ is a $*$-invariant linear subspace of $B(D,D)$ such that $a\cdot x\cdot b\in X$ for  $x\in X$ and $a,b\in A$.
Then $X$ is a $*$-bimodule for $A$ with definitions (\ref{defiaction}) and (\ref{defin}).
\end{cor}

The $*$-bimodules of sesquilinear forms described in Corollary \ref{substar} will be our guiding examples  algebraic model for $*$-bimodules.

\begin{exm}\label{singegen}
Let $x$ be a fixed sesquilinear form of $B(D,D)$ and let  $X_x$ be the linear span of elements $a\cdot x\cdot b$, where $a,b\in A$. Clearly, $X_x$ is a subbimodule of the $A$-bimodule $B(D,D)$. 
If there exist elements $c,d\in A$ such that $x^+=c \cdot x\cdot d$, then $X_x$ is $*$-invariant, so $X_x$ is the  $*$-bimodule for $A$ by Corollary \ref{substar}. In fact, $X_x$ is the $*$-bimodule generated by the form $x$. All this holds, for instance, if $x=x^+$.
\end{exm}

The $A$-bimodule operations on $B(D,D)$ depend essentially on the action of $A$ on $D$. In order to emphasize this action let us denote the left module action of $A$ on $D$ by $\sigma$ (that is, $\sigma(a)\varphi=a\cdot \varphi$). Then, by Proposition  \ref{algexa}, $B(D,D)$ becomes a $*$-bimodule for $A$ with left and right actions  $\cdot_\sigma$ defined by
\smallskip
\begin{align}\label{defin1}
(a\cdot_\sigma x)(\varphi,\psi)=x(\varphi,\sigma(a^+) \psi), \quad (x\cdot_\sigma b)(\varphi,\psi)=x(\sigma(b)\varphi,\psi),\quad \varphi,\psi\in \cD.
\end{align}

Now we are ready for another basic concept on $*$-bimodules.
\begin{dfn}\label{algrep}
Suppose  $D$ is a left $A_1$-module for a unital $*$-algebra $A_1$ and $B(D,D)$ is the corresponding  $*$-bimodule for $A_1$ with left and right actions  (\ref{defin}).

 An \emph{algebraic representation} of  a $*$-bimodule $X$ for the $*$-algebra $A$ on $D$ is a homomorphism $(\theta,\rho)$ of the $*$-bimodule $X$ for $A$ into the $*$-bimodule $B(D,D)$ for $A_1$ according to Definition \ref{homomprphism}, that is, $\rho:A\mapsto A_1$ is a unital $*$-homomorphism of $*$-algebras and $\theta:X\mapsto B(D,D)$ is a linear mapping  such that 
\begin{align}
\theta(x^+)(\varphi,\psi)=\overline{\theta(x)(\psi,\varphi)}\quad {\rm and}\quad
\theta (a\cdot x\cdot b)(\varphi, \psi)= \theta(x)(\rho(b)\cdot\varphi, \rho(a^+)\cdot\psi).
\end{align}for $a,b\in A, x\in X,$ and $\varphi, \psi\in D$.
\end{dfn}
\begin{rem}
Let $D$  be a left  module for  $A$ with two left  actions denoted by $\sigma_1$ and $\sigma_2$, respectively. Then $B(D,D)$  becomes an $A$-bimodule with definitions
\begin{align}\label{defin1}
(a\cdot x)(\varphi,\psi)=x(\varphi,\sigma_2(a^+) \psi), \quad (x\cdot b)(\varphi,\psi)=x(\sigma_1(b)\varphi,\psi),\quad \varphi,\psi\in \cD.
\end{align}
In order to prove that $B(D,D)$ is a $*$-bimodule we need that $\sigma_1=\sigma_2$. 
\end{rem}
\section{Hilbert space representations of $*$-bimodules}\label{Hilbrep}

In this section we pass from the algebraic considerations of the preceding section to the Hilbert space setting. For this we suppose that $A$ is a  {\it complex} unital $*$-algebra and $D$ is a {\it complex} inner product space $(\cD,\langle \cdot,\cdot\rangle)$. Let $\cH$ be its Hilbert space completion. Before we define $*$-representations of abstract $*$-bimodules we study $*$-bimodules of operators.

Let  $L(\cD,\cH)$ be the vector space of all linear operators $x:\cD\mapsto \cH$.
Let $\cL^+(\cD,\cH)$ denote the set of operators $a\in L(\cD,\cH)$  such that $\cD(a^*)\supseteq \cD$. Set $a^+:=a^*\lceil \cD$. It is not difficult to verify that  $\cL^+(\cD,\cH)$ is  a $*$-vector space with involution $a\mapsto a^+$. As usual in unbounded operator algebras, $\cL^+(\cD)$ denotes the set of operators $a\in \cL^+(\cD,\cH)$ such that $a$ and $a^*$ map the space $\cD$ into itself. Then $\cL^+(\cD)$ is a unital $*$-algebra with operator product as multiplication and involution $a\mapsto a^+$, see, for instance, \cite{schm1990}. 

Suppose that $t\in L(\cD,\cH)$. We define a sesquilinear form $x_t$ of $B(\cD,\cD)$ by 
\begin{align}\label{xt}
x_t(\varphi,\psi):=\langle t\varphi, \psi\rangle ,\quad \varphi,\psi\in \cD.
\end{align}
Further, if $t \in\cL^+(\cD,\cH)$, then we have $(x_t)^+=x_{t^+}$, that is, the operator $t^+$ corresponds to the adjoint form $(x_t)^+\in B(\cD,\cD)$ of $x_t$. The sesquilinear forms $x_t$ with $t \in\cL^+(\cD,\cH)$ form a $*$-vector space.

Now we suppose  $A$ is a {\it unital $O^*$-algebra} $\cA$ on $\cD$. This means that  $\cA$ is a $*$-subalgebra of the $*$-algebra $\cL^+(\cD)$ which contain the identity map $I$. The operators $a\in \cA$ act on vectors $\varphi\in \cD$ and the map $(a,\varphi)\mapsto  a\varphi$ is  a left action of the unital $*$-algebra $\cA$ on the vector space $\cD$. Hence we are in the setup of  the preceding section. 
 Let $\cB_\cA^+(\cD,\cD)$ denote the linear span of forms $a\cdot x_t\cdot b$, where  $a,b\in \cA$ and $t\in \cL^+(\cD,\cH)$. In the special case $\cA=\cL^+(\cD)$ we simply write $\cB^+(\cD,\cD)$ instead of $\cB_\cA^+(\cD,\cD)$. Clearly, $\cB^+_\cA(\cD,\cD)$ is $*$-invariant. Therefore, by Corollary \ref{substar} and equations (\ref{defin}), {\it $\cB^+_\cA(\cD,\cD)$ is a  $*$-bimodule for $\cA$ with actions given by }
\begin{align}\label{ostar} 
(a\cdot x_t\cdot b )(\varphi,\psi):=\langle tb\varphi,a^+ \psi\rangle, \quad a,b\in \cA,\, t\in \cL^+(\cD,\cH),\,\varphi,\psi\in \cD .
\end{align}

If the operator $t$ maps $\cD$ into itself, we have $\langle tb\varphi,a^+ \psi\rangle =\langle atb\varphi, \psi\rangle $, so $a\cdot x_t\cdot b$ is the sesquilinear form\, $x_{atb}$\, which is associated by (\ref{xt}) with the operator product $atb\in L(\cD,\cH).$ But in general,  the operator $t$ does not map $\cD$ into itself and the sesquilinear form $a\cdot x_t\cdot b$ cannot be given by an operator of $L(\cD,\cH)$.

Setting $t=I$ and $t=a\in \cA$, we obtain  $x_I(\varphi,\psi)=\langle \varphi, \psi\rangle$ and $x_a(\varphi,\psi)=\langle a\varphi, \psi\rangle$ from (\ref{xt}). Let us identify each operator $a\in \cA$ with the associated sesquilinear form $x_a(\varphi,\psi)=\langle a\varphi,\psi\rangle$  of\, $\cB^+_\cA(\cD,\cD)$. Then
 $x_{a^+}=(x_a)$ by definition and
$$(a\cdot x_I\cdot  b )(\varphi,\psi)=\langle I\ b\varphi,a^+ \psi\rangle =\langle ab\varphi, \psi\rangle =x_{ab}(\varphi,\psi\rangle.$$
by (\ref{ostar}). Therefore, under the identification of $a\in \cA$ with $x_a\in \cB^+_\cA(\cD,\cD)$, the involution and the multiplication of the $*$-algebra $\cA$ correspond to the involution and the actions of the $*$-bimodules $\cB^+_\cA(\cD,\cD)$. Summarizing, this  shows  that the {\it pair $(\cA,\cB^+_\cA(\cD,\cD))$ is a quasi $*$-algebra according to Definition \ref{quasidef}.}

Before we continue let us note that the weak commutant of the $O^*$-algebra $\cA$ can be nicely illustrated in this manner, as the following example shows.
\begin{exm}
Let $T$ be a bounded operator on $\cH$ and let $t:=T\lceil \cD$. Then  
\begin{align*}
(x_t\cdot a)( \varphi,\psi)=\langle ta \varphi,\psi\rangle=\langle Ta \varphi,\psi\rangle~~~ {\rm and} ~~~(a\cdot x_t)(\varphi,\psi)=\langle. t\varphi,a^+\psi\rangle =\langle T\varphi,a^+\psi\rangle ,
\end{align*}
Hence it follows that $x_t\cdot a=a\cdot x_t$ for all $a\in \cA$ if and only if the operator $T$ belongs to the weak commutant $\cA_w^\prime$.
\end{exm}

Another main notion of this paper is introduced in following definition. 
\begin{dfn}\label{rep}
Let $X$ be a $*$-bimodule for $A$.
Suppose $\rho$ is a nondegenerate $*$-representation of $A$ with domain $\cD$ and $\theta$ is linear map of $X$ into $\cL^+(\cD,\cH)$ which preserves the involution (that is, $\theta(x^+)=\theta(x)^+$ for $x\in X$). The pair $(\theta,\rho)$ is called a
 \emph{$*$-representation}  of  the $*$-bimodule $X$ on $\cD$ if 
\begin{align}\label{repcond}
\langle\theta (a\cdot x\cdot b)\varphi,\psi\rangle=\langle\theta(x)\rho(b)\varphi, \rho(a^+)\psi\rangle
\quad {\rm for}~~~a,b\in A, x\in X,\,  \varphi, \psi\in \cD.
\end{align}
 
 A $*$-representation $(\theta,\rho)$   of  $X$ is called \emph{faithful} if  $\rho$ and $\theta$ are injective, that is, $\rho(a)=0$ for $a\in A$ implies $a=0$ and $\theta(x)=0$ for $x\in X$ implies $x=0$.
\end{dfn}

First we   relate the $*$-representations of Definition \ref{rep} to the algebraic representations from Definition \ref{algrep}.
 
As in Definition \ref{rep}, let $\rho$ be a nondegenerate $*$-representation of $A$ on $\cD$ and $\theta$ a linear map of $X$ into $\cL^+(\cD,\cH)$.
Further, let $\cB^+_\cA(\cD,\cD)$ be the $*$-bimodule for the $O^*$-algebra $\cA:=\rho(A)$ defined above, with actions given by (\ref{ostar}).  
Then the pair $(\theta,\rho)$ is a $*$-representation  of  $X$ according to Definition \ref{rep} if and only if $(\theta,\rho)$ is a  homomorphism of the $*$-bimodule $X$ into the $*$-bimodule $\cB^+_\cA(\cD,\cD)$  in the sense of Definition \ref{homomprphism}, or equivalently, $(\theta,\rho)$ is an algebraic representation of the $*$-bimodule $X$ on the vector space $\cD$ according to Definition \ref{algrep}.

The following  is the notion of a {\it $qu*$-representation} of a quasi $*$-algebra, as defined in \cite[Definition 1.2.5]{tfr}.
\begin{dfn}
A \emph{$*$-representation} of  a quasi $*$-algebra  $(A,X)$ on a complex inner product space $\cD$ is a linear mapping $\theta$ of $X$ into $\cL^+(\cD,\cH)$ such that $\rho:=\theta\lceil A$ is a $*$-representation of the $*$-algebra $A$ such that $\rho(1)=I$ and 
\begin{align}\label{repqu}
\theta(x^+)=\theta(x)^+ \quad {\rm and}\quad\theta(x\cdot a)=\theta(x)\rho(a)\quad {\rm for}~~~ a\in A, x\in X.
\end{align}
\end{dfn}
The next proposition shows that this notion is a $*$-representation of the $*$-bimodule $X$ according to Definition \ref{rep}.
\begin{prop}
Suppose $\theta$ is a  $*$-representation of  a quasi $*$-algebra $(A,X)$ and let $\rho:=\theta\lceil A$. Then  the pair $(\theta,\rho)$ is a $*$-representation of the $*$-bimodule $X$ according to Definition \ref{rep}.
\end{prop}
\begin{proof} Let $a,b\in A, x\in X$ and $\varphi,\psi\in \cD$. Using (\ref{repqu}) we derive
\begin{align*}
&\langle \theta (a\cdot x\cdot b)\varphi,\psi\rangle=\langle \theta((a\cdot x)\cdot b)\varphi,\psi\rangle\\&=\langle \theta(a\cdot x)\rho( b)\varphi,\psi\rangle
=\langle \rho(b)\varphi,\theta(a\cdot x)^+\psi \rangle \\& =\langle \rho(b)\varphi,\theta((a\cdot x)^+)\psi \rangle=\langle \rho(b)\varphi,\theta(x^+\cdot a^+)\psi \rangle\\&=\langle \rho(b)\varphi,\theta(x^+)\rho(a^+)\psi\rangle =\langle \theta(x) \rho(b)\varphi,\rho(a^+)\psi\rangle,
\end{align*}
which proves  condition (\ref{repcond}).
\end{proof}
\begin{prop}\label{adjointrep}
If $(\theta,\rho)$ is a $*$-representation of a $*$-bimodule $X$, then we have $\rho(a)\cD\subseteq  \cD(\rho^*)$ and\,  $\theta (a\cdot x\cdot b)= \rho^*(a)\theta(x)\rho(b)$\, for $a,b\in A$ and $x\in X$.
\end{prop}
\begin{proof}
From equation (\ref{repcond}) it follows that the vector $\theta (x)\rho(b)\varphi$ belongs to the domain of the adjoint operator $\rho(a^+)^*$ for $a\in A$. The intersection of these domains for all $a\in A$ is the domain $\cD(\rho^*)$ of the adjoint representation. Therefore, we obtain  $\rho^*(a)=\rho(a^+)^*\lceil \cD(\rho^*)$. Hence the right-hand side of (\ref{repcond}) yields the equality $\theta (a\cdot x\cdot b)= \rho^*(a)\theta(x)\rho(b)$.
\end{proof}
\begin{dfn}\label{repstrong}
A $*$-representation $(\theta,\rho)$ of a $*$-bimodule $X$ is called \emph{strong} if $\theta(x)\in \cL^+(\cD)$ for all $x\in X$.
\end{dfn}
From Proposition \ref{adjointrep} it follows that if $\rho$ is self-adjoint, then $(\theta,\rho)$ is strong,  and if $(\theta,\rho)$ is strong, then we have
\begin{align*}
\theta (a\cdot x\cdot b)=\rho(a)\theta(x)\rho(b)\quad {\rm for}~~~a,b\in A, x\in X.
\end{align*}

Recall from Example \ref{generalex} that many $*$-bimodules are obtained from the following  situation: There is a larger $*$-algebra $B$ which contains $A$ as a $*$-subalgebra and $X$ as a $*$-subspace,  the actions of $A$ on $X$ are given by the multiplication of $B$, and the involutions of  $A$ and $X$ are the involutions in $B$. In this special case, it is clear that the restrictions of each $*$-representation of $B$ to $X$ and $A$ yield  a {\it strong $*$-representation} of the $*$-bimodule $X$. 
This emphasizes the importance the notion  of a strong $*$-representation.
\begin{exm}\label{weyl}
Let $B$ denote the one-dimensional Weyl algebra, that is, $B$ is the unital  $*$-algebra with hermitean generators $p,q$ and defining relation $pq-qp=-{\rm i}$. Let $A=\dC[q]$. Let $X_2$ and $X_1$ denote  the spans of elements $fp^2g$ and $fpg$, respectively,   with $f,g\in A$. Obviously, $X_2$ and $X_1$ are  $*$-bimodules for $A$, with algebraic operations inherited from the Weyl algebra $B$. For  simplicity we set $d={\rm i}p$. 
Then $X_2=\Lin \{fd^2g: f,g\in A\}$ and $X_1=\Lin \{fdg: f,g\in A\}$. 

First we show that there is a $*$-homomorphism $\vartheta$ of the $*$-bimodule $X_2$ on the $*$-bimodule $X_1$ such that $\vartheta (fd^2q)=-f dg$ for $f,g\in A$. The crucial step 
 is to show that the linear map $\vartheta$ is well-defined.  Each element  of $X_2$ is a finite sum $x=\sum_j f_jd^2g_j$ with $f_j,g_j\in A$. Suppose that $x= \sum_j f_jd^2g_j=0$ in $X_2$. Then, since $d^2g_j=g_jd^2+ 2g_j^\prime d+g_j^{\prime\prime}$, it follows that $\sum_f f_jg_j=\sum_j f_jg_j^\prime=\sum_j f_jg_j^{\prime\prime}=0$ in the  algebra $A=\dC[q]$.
Then $\sum_j f_jdg_j=\sum_j f_jg_jd + \sum_j f_jg_j^\prime=0$, so $\theta$ is well-defined. Since $d^+=-d$, it is clear that $\vartheta$ is a $*$-homomorphism of $*$-bimodules.

Note that $\vartheta$ is not injective. Indeed, $x=q^2d^2-2qd^2q+d^2q^2=4(q-q^2)d+2q\neq 0$, but  $\vartheta (x)=-q^2d+2qdq+dq^2=0.$

Let $\pi$ be the Schr\"odinger representation of the Weyl algebra $B$. It acts on $\cD(\pi)=\cS(\dR)$ in the Hilbert space $L^2(\dR)$ by $\pi(q)=t$ and $\pi(p)=-{\rm i}\frac{d}{dt}$, where $t$ denotes the variable of functions of $\cS(\dR)$. Hence $\theta_1(x)=\pi(x)$ and $\rho(a)=\pi(a)$ for $x\in X_1, a\in A$ defines a strong $*$-representation $(\theta_1,\rho)$ of the $*$-bimodule $X_1$. 

For $x\in X_2$ we define $\theta_2(x)=\theta_1(\vartheta(x))$, that is, $\theta_2(f p^2g)=\pi(f(q))\pi(p)\pi(g(q))$ for polynomials $f(q),g(q)\in A=\dC[q]$. Since $\vartheta:X_2\mapsto X_1$ is a $*$-homomorphism of $*$-bimodules,  $(\theta_2,\rho)$ is a  strong $*$-representation of $X_2$. By this definition, we have $$ \theta_2(-2\ii p)=\theta_2(p^2q-qp^2)=\pi(pq-qp)=\pi(-\ii 1)=-\ii I,$$
so that $\theta_2(p)=\frac{1}{2}I$. Since $pq-qp=-\ii $ in  $B$, this  implies that the $*$-representation $\theta_2$ of $X_2$ does not come from a $*$-representation of the Weyl algebra $B$.
\end{exm}
\begin{exm}\label{singr}
Suppose $\cA$ is a unital $O^*$-algebra on an inner product space $\cD$ and $t$ is a symmetric operator of $\cL^+(\cD,\cH)$. Recall $x_t$ and $a\cdot x_t\cdot b$, where $a,b\in A$, denote the sesquilinear forms of $B(D,D)$ defined by (\ref{xt}) and (\ref{ostar}), respectively.
Then  $X:=\Lin \, \{a\cdot x\cdot b: a,b\in \cA\}$\, is a $*$-subbimodule  of the $*$-bimodule $\cB^+_\cA(\cD,\cD)$ for $\cA$ defined above. Let $\rho$ denote the identity $*$-representation of $\cA$, that is, $\rho(a)=a$ for $a\in \cA$. Then, by (\ref{ostar}), there is a well-defined linear mapping $\theta$ of $X$ into $B(\cD,\cD)$ given by 
\begin{align}\label{singlerep}
\theta(a\cdot x_t\cdot b)(\varphi,\psi)=\langle t\rho(b)\varphi,\rho(a^+)\psi\rangle, \quad a,b\in A,~\varphi,\psi\in \cD.
\end{align}
Clearly, $(\theta,\rho)$ is an algebraic representation of the $*$-bimodule $X$ according to Definition \ref{algrep}. 
However, if the range of $t$ is not contained in the domain of the adjoint of the operator $\rho(a^+)$, then it follows from (\ref{singlerep}) that the form $a\cdot x_t\cdot I$ cannot be given by a Hilbert space operator. In this case,  $(\theta,\rho)$ is not a $*$-representation of the $*$-bimodule $X$ in the sense of  Definition \ref{rep}. If $t\in \cL^+(\cD)$, then it is obvious that $(\theta,\rho)$ is a strong $*$-representation of the $*$-bimodule $X$.
\end{exm}

Definitions \ref{rep} and \ref{repstrong} are two reasonable ways to define Hilbert space representations of $*$-bimodules. In both cases the elements of the image $\theta(X)$ are {\it operators}, in the first case of $\cL^+(\cD,\cH)$ and in the second case of $\cL^+(\cD)$.

Example \ref{singr} and equation (\ref{singlerep}) suggests another  possibility to define Hilbert space $*$-representations of  $*$-bimodules. Let $T$ be a distinguished generating set of a $*$-bimodule $X$, that is, $X$ is the smallest $*$-subbimodule  containing $T$. 
\begin{dfn}
 A \emph{weak $*$-representation} of $X$  on an inner product space $\cD$ (with respect to the generating set $T$) is a pair $(\theta,\rho)$ of a linear $*$-preserving mapping of $X$ into $B(\cD,\cD)$ and a nondegenerate $*$-representation $\rho$ of $A$ on $\cD$ such that $\theta(t)\in \cL^+(\cD,\cH)$ and
\begin{align*}
\theta(a\cdot t\cdot b)(\varphi,\psi)=\langle t\rho(b)\varphi,\rho(a^+)\psi\rangle\quad {\rm for}~~~ t\in T, a,b\in A, \varphi,\psi\in \cD.
\end{align*}
\end{dfn}
Obviously, each $*$-representation is a weak $*$-representation. For a weak $*$-representation, the image $\theta(t)$ of  $t\in T$ is an operator, but  in general $\theta(a\cdot t\cdot b)$\,  is only a {\it form} that is not necessarily given by an  operator.

We illustrate  Definition \ref{rep} in the case of the $*$-bimodule from Example \ref{singegen}.

\begin{exm}
Suppose $x$ is a hermitean sesquilinear form of $B(D,D)$
and  $X_x:=\Lin A\cdot x\cdot A$\,  denotes the corresponding $*$-bimodule  defined in Example \ref{singegen}. Let $(\rho,\theta)$ be a Hilbert space representation of the $*$-bimodule $X_x$ on an inner product space $(\cD,\langle\cdot,\cdot\rangle)$. Then, by Definition \ref{rep}, we have $\theta(x)\in \cB^+(\cD,\cH)$, say  $\theta(x)=x_t$ with $t\in \cL^+(\cD,\cH)$, and 
\begin{align}\label{singlerep}
\theta(a\cdot x\cdot b)(\varphi,\psi)=\langle t\rho(b)\varphi,\rho(a^+)\psi\rangle, \quad a,b\in A,~\varphi,\psi\in \cD.
\end{align}
That is, each Hilbert space representation of $X_x$ is completely determined by the $*$-representation $\rho$ of the $*$-algebra $A$ on $\cD$ and the operator 
$t\in \cL^+(\cD,\cH)$. 
\end{exm}
It is an   open problem to find reasonable sufficient conditions for a $*$-bimodule to have a faithful Hilbert space representation.

\section{A GNS-like construction for $*$-bimodules}\label{gnslike}

Suppose $f$ is a positive linear functional on the $*$-algebra $A$. First we recall  the definition of the GNS representation $\rho_f$ of $f$, see, e.g., \cite[Section 4.4]{schm2020} for details. 
Let $\cN_f:=\{a\in A: f(a^+ a)=0\}$ be the kernel of $f$.
Then the  $*$-representation $\rho_f$ acts on the quotient  space $\cD_f/\cN_f$, equipped with inner product $\langle a +\cN_f,b+\cN_f\rangle =f(b^+ a)$, by $\rho_f(a)(b+\cN_f)=ab+\cN_f$ and we have 
\begin{align}\label{gns}
f(a^+ b) =\langle \rho_f(b)\varphi_f,\rho_f(a)\varphi_f\rangle, \quad {\rm where}~~a,b\in A.
\end{align}
The equivalence class $\varphi_f:=1+\cN_f$ is an algebraically cyclic vector for $\rho_f$, that is, $\cD_f=\rho_f(A)\varphi_f$. Note that $\rho_f(1)\varphi_f=\varphi_f$.
Let $\cH_f$ denote the Hilbert space completion of the inner product space $\cD_f$.
\smallskip

Now we suppose that $X$ is a $*$-bimodule for $A$ and $F$ is a linear functional on $X$. 
We want to represent $F$ in the form 
\begin{align}\label{gns0}
F(a\cdot x\cdot b)=\langle \theta(x)\rho(b)\varphi,\rho(a^+)\varphi\rangle, \quad a,b\in A,~ x\in X,
\end{align}
 for some $*$-representation $(\theta,\rho)$ of $X$ and some vector $\varphi$ of the domain. 
 
First suppose that (\ref{gns0}) holds.
 We show that  $F$ is hermitian. Indeed,let $x\in X$. Then, since  $\rho(1)\varphi=\varphi$,
\begin{align*}
F(x^+)&=F(1\cdot x^+\cdot 1)=\langle \theta(x^+)\rho(1)\varphi,\rho(1)\varphi \rangle \\&=\langle  \varphi,\theta(x)\varphi\rangle=\overline{\langle \theta(x) \varphi,\varphi\rangle}=\overline{F(x)}.
\end{align*}
Let $f$ denote the positive functional on $A$ defined by $f(a):=\langle\rho(a)\varphi,\varphi\rangle, a\in A$. Then, by (\ref{gns0}), we obtain for $a\in A$,
\begin{align}\label{Ffesti}
|F(a^+\cdot x)|^2&=|F(a^+\cdot x\cdot 1)|^2=|\langle \theta (x)\rho(1)\varphi,\rho(a)\varphi\rangle|^2\nonumber\\&= |\langle\theta (x)\varphi,\rho(a)\varphi\rangle|^2 \leq \|\theta(x)\varphi\|^2\, \|\rho(a)\varphi\|^2 =  \|\theta(x)\varphi\|^2 f(a^+a).
\end{align}

The following theorem   shows that for any hermitian functional $F$ on $X$ satisfying an inequality of the form (\ref{Ffesti}) there is 
a representation such that (\ref{gns0}) holds.
This is one possible way to extend the  GNS construction   to $*$-bimodules.

\begin{thm}\label{GNSB}
Let $F$ be a hermitian linear functional on the $*$-bimodule $X$  and let $f$ be a positive  functional on $A$. Suppose that
for each element $x\in X$ there exists a constant $C_x>0$ such that
\begin{align}\label{assu}
|F(a^+\cdot x)|^2\leq C_x f(a^+a)\quad {\rm for}~~~a\in A.
\end{align}
 Let $\rho_f$ be the GNS representation of  $f$. 
Then there exists a $*$-representation $(\theta_F,\rho_f)$ of the $*$-bimodule $X$ on  $\cD_f$ such that 
\begin{align}\label{gns1}
F(a\cdot x\cdot b)=\langle \theta_F(x)\rho_f(b)\varphi_f,\rho_f(a^+)\varphi_f\rangle \quad {\rm for}~~a,b\in A,~ x\in X.
\end{align}

If $(\theta,\rho)$ is another $*$-representation of the $*$-bimodule $X$ on an inner product space $(\cD,(\cdot,\cdot))$ with algebraically cyclic vector $\psi$  for $\rho$ such that $f(a)=\langle \rho(a)\psi\rangle, a\in A$, and 
\begin{align}\label{unigns}
F(a\cdot x\cdot b)=( \theta(x)\rho(b)\psi,\rho(a^+)\psi)\quad {\rm for}~~~ a,b\in A, x\in X,
\end{align} then the $*$-representations $(\theta_F,\rho_f)$ and $(\theta,\rho)$ are unitarily equivalent. 
There is a unitary operator $U$ of  $\cH_f$ on the Hilbert space completion $\cH$ of $\cD$ such that $U\varphi_f=\psi$, $U\cD_f=\cD$,  $\rho(a)=U\rho_f(a)U^{-1}$, and\, $\theta(x)=U\theta_F(x)U^{-1}$ for $a\in A, x\in X$.
\end{thm}
\begin{proof} 
This proof follows as the same pattern as the proof of the GNS representation for quasi $*$-algebras (see e.g. \cite[Theorem 1.4.8]{tfr}). However, the case of $*$-bimodules requires much more care and various modifications.

Let  $x\in X$ and $b\in A$. Then, by   assumption (\ref{assu}) and equation (\ref{gns}), we have 
\begin{align*}
|F(x\cdot b)|^2=|F((x\cdot b)^+)|^2=|F(b^+\cdot x^+)|^2\leq C_{x^+} f(b^+b)=C_{x^+}\|\rho_f(b)\varphi_f\|^2.
\end{align*}This implies that the map $\rho_f(b)\varphi_f\mapsto F(x\cdot b)$ is a (well-defined!) linear functional on the dense linear subspace $\cD_f=\rho_f(A)\varphi_f$ of the Hilbert space $\cH_f$ which is continuous in the Hilbert space norm. Therefore, by a classical result of F. Riesz, there exists a unique vector $\xi_x\in \cH_f$, depending on $x$, such that 
\begin{align}\label{f1b}F(x\cdot b)=\langle \rho_f(b)\varphi_f,\xi_x\rangle\quad {\rm for}~~ b\in A.
\end{align}
Then, for $a\in A$,
\begin{align}\label{f1a}
F(a\cdot x)=\overline{F(x^+\cdot a^+)}=\overline{\langle \rho_f(a^+)\varphi_f,\xi_{x^+}\rangle} =\langle \xi_{x^+},\rho_f(a^+)\varphi_f\rangle
\end{align}
Applying (\ref{f1b}) and (\ref{f1a}) to the bimodule equality $$F(a \cdot x\cdot b)=F((a\cdot x)\cdot b)=F(a\cdot (x\cdot b))$$ we obtain
\begin{align}\label{fab}
F(a\cdot x\cdot b)=\langle \rho_f(b)\varphi_f,\xi_{a\cdot x}\rangle=\langle \xi_{(x\cdot b)^+},\rho_f(a^+)\varphi_f\rangle
\end{align}
From this equality and the fact that $\cD_f=\rho_f(A)\varphi_f$ is dense in $\cH_f$ it follows that  $\rho_f(b)\varphi_f=0$ always implies that $\xi_{(x\cdot b)^+}=0$.
Therefore, for any $x\in X$, there is a well-defined (!) linear map $\theta_F(x)$ of $\cD_f$ into $\cH_f$ given by\, $\theta_F (x)\varphi_f(b)\varphi_f=\xi_{(x\cdot b)^+}$. Then
$\theta_F(x^+)\rho(a^+)\varphi_f=\xi_{(x^+\cdot a^+)^+}=\xi_{a\cdot x}$. Inserting these facts into (\ref{fab}) yields
\begin{align}\label{fab2}
F(a\cdot x\cdot b)=\langle \rho_f(b)\varphi_f,\theta_F(x^+)\rho(a^+)\varphi_f\rangle=\langle \theta_F (x)\varphi_f(b)\varphi_f,\rho_f(a^+)\varphi_f\rangle
\end{align}
for all $a,b\in A$. Clearly, the map $x\mapsto \theta_F(x)$ is linear. From (\ref{fab2}) we conclude that $\theta(x)\in \cL^+(\cD_f,\cH_f)$ and that $ \theta_F(x^+)=\theta_F(x)^+$ for any $x\in X$.

Next we verify that $\theta_F(x_1)\varphi_f=\varphi_f$. Let $a\in A$. Using (\ref{gns}) and (\ref{f1a}), we obtain
$$
\langle \xi_{x_1^+},\rho_f(a^+)\varphi_f\rangle=F(a\cdot x_1)=F(1^+a\cdot x_1)=\langle \rho_f(a)\varphi_f,\rho_f(1)\varphi_f\rangle=\langle \varphi_f,\rho_f(a^+)\varphi_f\rangle.
$$
Since $\rho(A)\varphi_f=\cD_f$ is dense in $\cH_f$, this implies $\varphi_f=\xi_{x_1^+}$. By construction, we have $\theta_F(x_1)\varphi_f=\theta_F(x_1)\rho_f(1)\varphi_f=\xi_{(x_1\cdot 1)^+}=\xi_{x_1^+}$, so that $\theta_F(x_1)\varphi_f=\varphi_f$.

Let $\eta,\zeta\in \cD_f$. Then there are elements $c,d\in A$ such that $\eta=\rho_f(c^+)\varphi_f$ and $\zeta=\rho(d)\varphi_f$. Applying the equality (\ref{fab2}) twice, first to the triple $\{c,a \cdot x\cdot b, d\}$ and then to\, $\{ca, x, bd\}$,
 we derive \begin{align*}
&\langle \theta_F(a \cdot x \cdot b)\zeta,\eta \rangle=\langle \theta_F(a \cdot x \cdot b)\rho_f(d)\varphi_f,\rho_f(c^+)\varphi_f \rangle\\&=F(c\cdot(a\cdot x\cdot b)\cdot d)=
F((ca)\cdot x\cdot (bd))=\langle \theta_F (x)\varphi_f(bd)\varphi_f,\rho_f((ca)^+)\varphi_f\rangle\\&=\langle \theta_F (x)\varphi_f(b)\varphi_f(d)\varphi_f,\rho_f(a^+)\rho_f(c^+)\varphi_f\rangle
=\langle \theta_F (x)\varphi_f(b)\zeta,\rho_f(a^+)\eta\rangle.
\end{align*} This shows that the pair $(\theta_F,\rho_f)$ is a Hilbert space representation of the $*$-bimodule $X$ and equation (\ref{gns1}) holds.

Finally, we prove  the uniqueness assertion. Since
\begin{align*}
\langle \rho_f(a)\varphi_f,\rho_f(a)\varphi_f\rangle =\langle \rho_f(a^+a)\varphi_f,\varphi_f\rangle= f(a^+a)=(\psi,\rho(a^+a)\psi)
=(\rho(a)\psi,\rho(a)\psi),
\end{align*}
 there is a well-defined isometric linear operator $U$ of the inner product space $\cD_f$ on the inner product space $\cD$ defined by $U(\rho_f(a)\varphi_f)=\rho(a)\psi, a\in A.$  We extend $U$ to a unitary operator of $\cH_f$ on $\cH$. The definition of $U$ in turn implies that $U\varphi_f=\psi$, $U\cD_f=\cD$, and $\rho(a)=U\rho_f(a)U^{-1}$ for $a\in A$. 

It remains to verify that $\theta(x)=U\theta_F(x)U^{-1}$ for  $x\in X$. Using the definition of the unitary operator $U$  
and equations (\ref{gns1}) and (\ref{unigns}) we derive
\begin{align*}
&(U\theta_F(x)U^{-1}(\rho(b)\psi), \rho(a^+)\psi)=(U\theta_F(x)\rho_f(b)\varphi_f,U(\rho_f(a^+)\varphi_f))\\ &=\langle\theta_F(x)\rho_f(b)\varphi_f,\rho_f(a^+)\varphi_f\rangle= F(a \cdot x\cdot b)=(\theta(x)(\rho(b)\psi), \rho(a^+)\psi).
\end{align*}
Since $\psi$ is algebraically cyclic for $\rho$ by assumption, $\rho(A)\psi=\cD$ is dense in $\cH$. Hence the preceding equality yields  $U\theta_F(x)U^{-1}\rho(b)\psi=\theta(x)\rho(b)\psi$ for all $b\in A$, so that $\theta(x)=U\theta_F(x)U^{-1}.$ 
 This completes the proof of Theorem \ref{GNSB}.
\end{proof}

Next we want to characterize the case when an operator $\theta_f(x)$, $x\in X$, is bounded.
First  we prove an auxiliary operator-theoretic lemma.
\begin{lem}\label{aux}
Let  $T$ be  a linear   operator on a Hilbert space. If there exists a constant $M\geq 0$ such that
$|\langle T\eta,\eta\rangle| \leq M\|\eta\|^2$ for all $\eta\in \cD(T)$, then  $T$ is bounded.
\end{lem}
\begin{proof}
Let $\eta,\zeta\in \cD(T)$. By the polarization identity we have
\begin{align*}
4\langle T\eta, \zeta\rangle = &\langle T(\eta+\zeta),\eta+\zeta\rangle-\langle T(\eta-\zeta),\eta-\zeta\rangle \\ &+\ii
\langle T(\eta+\ii\zeta),\eta+\ii\zeta\rangle
-\ii \langle T(\eta-\ii\zeta),\eta-\ii\zeta\rangle.
\end{align*}
From this identity and the assumption we obtain 
\begin{align*}
|4\langle T\eta,\zeta\rangle|& \leq M\big( \|\eta+\zeta\|^2+\|\eta-\zeta\|^2+\|\eta+\ii\zeta\|^2 +\|\eta-\ii\zeta\|^2\big)\\&\leq 
4M \big(\|\eta\|+\|\zeta\|\big)^2
\end{align*}
In particular, if $\|\eta\|\leq 1$ and $ \|\zeta\|\leq 1$, we have $|\langle T\eta,\zeta\rangle|\leq 4M$. Hence, by scaling the vectors $\eta,\zeta$  it follows that $|\langle T\eta,\zeta\rangle|\leq 4M\|\eta\|\, \|\zeta\| $ for arbitrary  $\eta,\zeta\in \cD(T)$. Finally, taking the supremum over $\zeta\in \cD(T), \|\zeta\|=1$, we get $\|T\eta\|\leq4M \|\eta\|$. Thus,  $T$ is bounded.
\end{proof}

\begin{prop}\label{buononeop}
Retain the assumptions and the notation of Theorem \ref{GNSB}. Let $x\in X$. Then the operator $\theta_F(x)$ is bounded on $\cD_f$ if and only if there exists a number $\lambda>0$ such that
\begin{align}\label{boundct}
F(a^+\cdot x\cdot a)\leq \lambda f(a^+a)\quad {\rm for}~~a\in A.
\end{align}
\end{prop}
\begin{proof}
First we suppose that condition (\ref{boundct}) is satisfied. Then, using  (\ref{gns1}), (\ref{boundct}) and (\ref{gns}) we obtain
\begin{align}
|\langle \theta(x)\rho_f(a)\varphi_f,\rho_f(a)\varphi_f \rangle|=|F(a^+\cdot x\cdot a)|\leq \lambda f(a^+a)=\lambda_x \|\rho_f(a)\varphi_f\|^2
\end{align}
for  $a\in A$. Thus, $|\langle \theta_F(x)\eta ,\eta\rangle|\leq \lambda \|\eta\|^2$ for all $\eta\in \rho_f(A)\varphi_f=\cD(\rho_f)$. Therefore, by Lemma \ref{aux}, the operator $\theta_F(x)$ is bounded.

The converse direction follows by reserving this reasoning.
\end{proof}\begin{dfn}\label{boundF}
Let $F$ and $f$ be as in  Theorem \ref{GNSB}.
We shall say that $F$ is  \emph{bounded} with respect to $f$ if for each $x\in X$ there exists a number $\lambda_x>0$ such that
$$
F(a^+\cdot x\cdot a)\leq \lambda_x f(a^+a)\quad {\rm for}~~a\in A.
$$
\end{dfn} 

Combining Proposition \ref{buononeop} and Definition \ref{boundF} yields the following corollary.

\begin{cor}
Let $F$ and $f$ be as in  Theorem \ref{GNSB}. Then all operators $\theta_F(x),$ $x\in X$, are bounded on $\cD_f$ if and only if the functional $F$ is bounded with respect to $f$.
\end{cor}

We illustrate this with a simple example. 

\begin{exm}
Let $A=\dC[x]$.   Then $X=e^{-x^2}\dC[x]$ is a $*$-bimodule for $A$, see Example \ref{exminusx2}. 
 Let $f$ be a real-valued Borel function on $\dR$ and  $\mu$  a Radon measure on $\dR$ such that the functions $fp$, $f^2p$,  $p$ are $\mu$-integrable for all $p\in A$. We define linear functionals $F$ on $X$ and $f$ on $A$ by
 $$
 F(e^{-x^2}p)=\int f(x)p(x)d\mu(x)\quad {\rm and}~~~ f(p)=\int p(x) d\mu(x), \quad   p\in A.$$ 
 Then the assumptions of  Theorem \ref{GNSB} are fulfilled. We have 
 $\rho_f(p)q=pq$ and $\theta_F(e^{-x^2}p)q= fpq$  for $q\in \cD_f\subseteq L^2(\dR;\mu)$, $p\in A$, and equation (\ref{gns}) holds.

 By appropriate choices of the measure $\mu$ and the function $f$ we see that the following four cases  can happen: \\
$\bullet$~  All operators $\theta_F(x)$ and $\rho_f(p)$  are bounded.\\
$\bullet$~  All operators $\theta_F(x)$ are bounded and all operators $\rho_f(p), p\neq 0,$ are unbounded.\\
$\bullet$~  All operators $\theta_F(x)$, $x\neq 0$, are unbounded and all operators $\rho_f(p)$ are bounded.\\
$\bullet$~  All operators $\theta_F(x)$, $x\neq 0$, and  $\rho_f(p), p\neq 0,$ are unbounded.  
\end{exm}

\section{An Example Concerning Theorem \ref{GNSB}}\label{anexampleweyl}

Let us  retain the notation of Example \ref{weyl}. Recall that $A$ is the $*$-subalgebra $\dC[q]$ of the Weyl algebra $B$ and we set $d={\rm i}p$. We consider the $*$-bimodule $X\subseteq B$ for $A$ defined by
\begin{align*}
X=\Lin \{ ap^2 b: a,b\in \dC[q]\}.
\end{align*}

Let $\mu$ be a Radon measure on $\dR$ such that all polynomials are $\mu$-integrable. To avoid trivial cases we assume that $\mu$ is not supported by a finite set. Let $f$ denote the positive linear functional on $A$ given by
$$
f(a)=\int a(x) d\mu(x), \quad a\in A=\dC[q].
$$

Suppose $x=\sum_j a_jd^2 b_j =0$ in $X$. Since $da-ad=a^\prime$ in the Weyl algebra $B$, we then have  $x=\sum_j ( a_jb_j^{\prime\prime}+ a_jb_j^\prime d + a_jb_j d^2) =0$ in  $B$ and therefore
\begin{align}\label{compat}
\sum\nolimits_j a_jb_j^{\prime\prime}=\sum\nolimits_j a_jb_j^\prime=\sum\nolimits_j  a_jb_j =0.
\end{align}
From this fact it follows that there are well-defined linear functionals $F_0,F_1, F_2$ on $X$ defined by
\begin{align*}
F_0(x)&:=\int \Big(\sum\nolimits_j a_j(x)b_j(x)\Big) d\mu(x),\\ F_1(x)&:=\int \Big(\sum\nolimits_j a_j(x)b_j^\prime(x)\Big) d\mu(x),\\F_2(x)&:=\int \Big(\sum\nolimits_j a_j(x)b_j^{\prime\prime}(x)\Big) d\mu(x)
\end{align*} 
for
$$x=\sum\nolimits_j  a_jd^2b_j\in X, \quad a_j,b_j\in A.$$

We show that assumption (\ref{assu}) is fulfilled for $F_0, F_1,$ and $F_2$. We carry out this proof for $F_0$. Let $x=\sum_j a_j d^2 b_j\in X$, $a\in A$ and set $h:=\sum_j a_jb_j\in A$. Using the Cauchy-Schwarz inequality for the positive functional $f$ we derive
\begin{align*}
|F_0&(a^+\cdot  x)|^2= \Big|\sum\nolimits_j F_0( aa_j d^2 b_j)\Big| =\Big|\int \Big(\sum\nolimits_j a(x)a_j(x)b_j(x)\Big) d\mu\Big|^2 \\ &= \Big|\int a(s)h(x) d\mu\Big|^2= |(ah)|^2\leq f(a^+a)f(h^+h),
\end{align*}
which gives (\ref{assu}) with $C_x=f(h^+h)$ for $F_0$. Replacing $h$ by $\sum_j a_jb_j^\prime$ and $\sum_j a_jb_j^{\prime\prime}$, respectively, the cases of $F_1$ and $F_2$ are treated in a similar manner.

Now we turn to the GNS representation $\rho_{f}$ of $f$.
Clearly, the inner product of   $\cD_{f}$ is the inner product inherited from  $L^2(\dR;\mu)$. (Since $\mu$ has no finite support,   $\cN_{f}=\{0\}$ and
we do not need equivalence classes.)
The GNS representation $\rho_{f}$ of $A$ acts on  $\cD_{f}=\dC[x]$ by $\rho_{f}(a)b=ab$ and we have  $\varphi_{f}=1.$ Then $f(a)=\langle \rho_{f}(a)\varphi_{f},\varphi_{f}\rangle$ for $a\in A$. 

For $x=\sum_j a_jd^2b_j\in X$ and $b\in A$ we define
\begin{align*}
\theta_{F_0}&\Big(\sum\nolimits_j a_jd^2b_j\Big)b\varphi_{f}=\sum\nolimits_j a_jb_jb \varphi_{f},\\ \theta_{F_1}&\Big(\sum\nolimits_j a_jd^2b_j\Big)b\varphi_{f}=\sum\nolimits_j a_j(b_jb)^\prime \varphi_{f},\\ \theta_{F_2}&\Big(\sum\nolimits_j a_jd^2b_j\Big)b\varphi_{f}=\sum\nolimits_j a_j(b_jb)^{\prime\prime} \varphi_{f}.
\end{align*} 
Since $x=0$ in $X$ implies the equations in  (\ref{compat}), these definitions give  well-defined (!) linear maps\, $\theta_{F_k}:X\mapsto \cL^+(\cD_{f})$ for $k=0,1,2$.  It is easily verified that the three pairs $(\theta_{F_k},\rho_{f})$, $k=0,1,2,$ are {\it strong $*$-representations} of the $*$-bimodule $X$ on $\cD_{f}$.

Inserting the preceding definitions yields for $x=\sum_j a_jd^2b_j\in X$ and $a,b\in A$, 
\begin{align*}
F_0(a\cdot x\cdot b)&=\int\Big(a(x)a_j(x)b_j(x) b(x)\Big) d\mu=\langle \theta_{F_0}(x)\rho_{f}(b)\varphi_{f},\rho_{f}(a^+)\varphi_{f}\rangle,\\ F_1(a\cdot x\cdot b)&=\int\Big(a(x)a_j(x)(b_jb)^\prime (x) \Big) d\mu=\langle \theta_{F_1}(x)\rho_{f}(b)\varphi_{f},\rho_{f}(a^+)\varphi_{f}\rangle,\\ F_2(a\cdot x\cdot b)&=\int\Big(a(x)a_j(x)(b_jb)^{\prime\prime}(x) \Big) d\mu=\langle \theta_{F_2}(x)\rho_{f}(b)\varphi_{f},\rho_{f}(a^+)\varphi_{f}\rangle.
\end{align*}
This shows that in all three cases equation (\ref{gns1}) is satisfied, so  the preceding formulas give the GNS-like representation of the pair of functionals $(F_k,f)$. For the generator $d^2$ of the $*$-bimodule $X$ we have in particular 
\begin{align*}
\theta_{F_0}(d^2)=I, ~~ \theta_{F_1}(d^2)=\frac{d}{dx},~~ \theta_{F_2}(d^2)=\Big(\frac{d}{dx}\Big)^2.
\end{align*}

It should be emphasized that the linear functionals $f$,  so its GNS representation $\rho_f$, and $F_k$, so   the mapping $\theta_{F_k}$, essentially depends on the measure $\mu$.  Further, the same functional $f$ can used for all three functionals $F_0, F_1, F_2$ on $X$ to satisfy assumption  (\ref{assu}).

\bibliographystyle{amsalpha}

\end{document}